\documentclass[article]{amsart}
\usepackage{amssymb}
\usepackage{color}
\usepackage{mathtools}
\usepackage[bookmarks, bookmarksdepth=2, colorlinks=true, linkcolor=blue, citecolor=blue, urlcolor=blue]{hyperref}
\usepackage{cleveref}
\usepackage{algorithm}
\newtheorem{theorem}{\bf Theorem}

\newtheorem{corollary}[theorem]{\bf Corollary}

\newtheorem{example}[theorem]{\bf Example}
\newtheorem{lemma}[theorem]{\bf Lemma}

\theoremstyle{definition}
\newtheorem{definition}[theorem]{\bf Definition}
\newtheorem{remark}[theorem]{\bf Remark}

\newcommand \N{{\mathbb N}}
\newcommand \Z{{\mathbb Z}}

\usepackage{color}
\definecolor{red}{rgb}{1.00,0.00,0.00}

\newcommand \0{{\mathbf{0}}}

\def \a{{\mathbf{a}}}
\def \b{{\mathbf{b}}}
\def \c{{\mathbf{c}}}

 \def \r{{\mathbf{r}}}
\def\sa{\langle {\bf a} \rangle}
\def\sb{\langle {\bf b} \rangle}
\def\sc{\langle {\bf c}\rangle}

\def\adj{\operatorname{Adj}}

\begin{document}
\title{Principal Matrices of Numerical Semigroups}
\author{Papri Dey and Hema Srinivasan }
\address{Department of Mathematics, University of Missouri, Columbia, MO 65211}
\maketitle
\begin{abstract} 
Principal matrices of a numerical semigroup of embedding dimension $n$ are special types of $n\times n$ matrices over the integers of rank $\le n-1$.  We show that such matrices as well as  all  pseudo principal matrices of size $n$  must have rank $\ge {\frac {n}{2}}$. We give structure theorems for pseudo principal matrices for which at least one $n-1\times n-1$ principal minor vanishes and thereby characterize the semigroups in embedding dimensions 4 and 5 in terms of their principal matrices.  We also prove a sufficient condition for an $n\times n$ pseudo principal matrix of rank $n-1$ to be principal.  
\end{abstract}

\section{Introduction}

Consider a sequence $  \a  =\{a_{1},\dots,a_{n}\} $ of positive integers and the monoid $\langle \a \rangle$ generated by $\a$. When the  $gcd (a_{1},\dots,a_{n})=1$,   $\langle \a \rangle$ is called a \textit{numerical semigroup} which  contains zero and all but finitely many positive integers.
Let $k$ be an arbitrary field. Consider the $k$- algebra homomorphism $\phi_{\a}:k[x_{1},\dots,x_{n}] \rightarrow k[t]$ given by $\phi_{\a}(x_{i})=t^{a_{i}}$. The image of this map $\phi_{\a}$ is the semigroup ring $k[\a]$ which is isomorphic to the coordinate ring $k[\a]= k[x_{1},\dots,x_{n}]/ I_{\a}$ where $\ker \phi_{\a}=I_{\a}$.  The semigroup ring $k[\a]$ and the ideal $I_{\a}$ become homogeneous with the weighting $\deg(x_i) = a_i, 1\le i\le n$.  The ideal $I_{\a}$ is also the toric ideal in $k[x_1, \dots, x_n]$ and is  generated by a finite set of binomials in $x_i$.  Since the semigroup ring is an integral domain of dimension one, the ideal $I_{\a}$ is a prime ideal of height $n-1$.  The embedding dimension of the semigroup ring is $n$ if and only if $I_{\a} \subset m^2$, $m= (x_1, \dots, x_n)$.   A semigroup is called complete intersection or Gorenstein if the corresponding semigroup ring is complete intersection or Gorenstein respectively. 

Given $  \a  =\{a_{1},\dots,a_{n}\}\subset \N$, for each $i, 1\le i\le n$, there exists a smallest positive integer  $r_{i}$ such that $r_{i}a_{i}=\sum_{j \neq i}r_{ij}a_{j}$.  The $n \times n$ matrix $D(\a):= (r_{ij})$ where $r_{ii}:= -r_{i}$ is called a \textit {principal matrix} associated to $\a$.
 
Although the diagonal entries $-r_i$ are uniquely determined, $r_{ij}$ are not always unique.  In that sense, there can be more than one principal matrix $D(\a)$ for a given $\a$.    Since $D(\a) a^T = 0$,  rank of $D(\a) \leq n-1$.  The sequence of positive integers $\a$ can be recovered from a given principal matrix $A$ of rank $n-1$  by factoring out the g.c.d of the entries of any nonzero column of the adjoint of $A$ and changing the signs to be positive.  Each of the $n$ rows of the principal matrix, corresponds to a binomial $\alpha_i, 1\le i\le n$, known as critical binomials \cite{rafael} as they are part of a minimal generating set of $I_{\a}$. Indeed, the principal matrix will have maximal rank $n-1$, if the  $(\alpha_i,1\le i\le n, x_j)$ is $m$-primary.  As such, the principal matrix of a numerical semigroup contains some of the essence of the information about the semigroup.  This article grew out of our attempts to uncover information about the semigroup or its structure from the principal matrix. 

This leads to the notion of a pseudo principal matrix, which looks like a principal matrix of $\a$, except that the diagonal entries $-r_i$ are not the smallest possible. We prove in Corollary \ref{rank} that an $n\times n$ pseudo principal matrix of size $n$ must have rank at least ${\frac{n}{2}}$.  Similar concepts have been studied in \cite{braz} and a related notion of generalized critical binomial matrix in \cite{Villarrealdegree}.  

We determine the structure of a pseudo principal matrix depending on its rank and use it to characterize the numerical semigroups $\langle \a \rangle$ from their principal matrices of any size in Theorem \ref{structure} and Theorem \ref{rkn-1}.  In particular, numerical semigroups of embedding dimension $4$ and $5$ have been discussed in terms of their principal matrices in detail in section 3.
For the embedding dimension $n=4$ we prove that the principal matrix $D(\a)$ has rank $2$ if and only if $\a= \{a_1, a_2, a_3, a_4\} = d\{c_1,c_2\}\sqcup e\{c_3,c_4\}$ and $D(\a)$ is the concatenation of two $2\times 2$ matrices  $D(c_1,c_2)$ and $D(c_3,c_4)$. 
For the embedding dimension $n=5$ we prove that the principal matrix $D(\a)$ has rank $3$ if and only if either $\a= \{a_1, a_2, a_3, a_4,a_5\} = d\{c_1,c_2,c_3\}\sqcup e\{c_4,c_5\}$ and $D(\a)$ has $3+2$ block structure or $\a= \{a_1, a_2, a_3, a_4,a_5\} = a_{1} \sqcup d\{c_2,c_3\}\sqcup e\{c_4,c_5\}$ and $D(\a)$ has $1+2+2$ block structure. On the other hand, we show that when $n=4$, if the rank of $D(\a)$ is $3$, then there are exactly three possibilities for $A= D(\a)$:(i) $A^TX=0$ has a solution in positive integers or (ii) $\a = d\{c_1, c_2\} \sqcup  \{b_1, b_2\}$ with $b_i \in \langle {c_1, c_2} \rangle$ or (iii).
 $\a = d\{b_1, b_2,b_3\} \sqcup b_4$ is a simple-split gluing and is either a complete intersection or an almost complete intersection.  

Furthermore, we also give a sufficient condition for a pseudo principal matrix of rank $n-1$ to be a principal matrix, in Proposition \ref{pseudoisreal}. This is a direct generalization to $n$ of the $4\times 4$ case in \cite{Be} and \cite{braz}.  

As a result of the general result in embedding dimension $4$ one finds that there is a finite number $\nu(\a)$ of a set of positive integers minimally generating a semigroup with a given principal matrix up to gluing. In the last section \ref{secexamples}, we list some examples to illustrate many of the results. 

\section{Principal matrix and look alikes}

For $I\subset \N$, $|I|$ denotes the number of elements in $I$. 
We write $[r]$ for the set  $\{1, 2, \ldots , r\}$  of first $r$ positive integers and   $|x|$ for the absolute value of $x$. Determinant of a matrix $T$ is denoted by  $\det T$ and the adjoint or adjugate of $T$ as $\adj T$. 
For any set $I$ we write $I -\{a\}$ is the subset of all elements in $I$ except $a$.  
When the context is clear we write $\a$ for the set $\{ a_1, \ldots , a_n\}$ as well as for the vector $\a\in \N^n$. 

 Let $A$ be a matrix.  Given $I, J\subset \N$, with $|I| = |J| $,   $A^I_J $ is the determinant of the submatrix of $A$ consisting of the rows in $I$ and the columns in $J$.
 If $I\subset \{1,2,\cdots, n\}$, then $\Delta_I$ denotes the determinant of the submatrix of $A$ consisting of the rows and columns in $I$. Thus, $\Delta_I = A^I_I$.   We denote the gcd of two integers $a$ and $b$ by  $(a,b)$.

 \begin{definition}
 Let $\a=\{ a_1, \ldots, a_n\}$ be a set of positive integers minimally generating a semigroup $\sa$. 
 $A= \begin{pmatrix}
 -c_1&a_{12}& \ldots & a_{1n}\\
 \vdots& \vdots& \vdots\\
 a_{n1} & a_{n2}& \vdots &-c_n\\
 \end{pmatrix}
 $ is called a principal matrix of $\a$ if $A\a = 0$ and $c_i$ is the smallest positive integer such that $c_ia_i \in \langle{\a -\{a_{i}\}}\rangle, 1\le i\le n\}$.  We denote this by $A = D(\a)$.  
 \end{definition}
 
 The relations $c_ia_i = \sum_{j\neq i} a_{ij}a_j$ are also called critical relations and the associated binomials $x_i^{c_i}-\prod_{j\neq i} x_j^{a_{ij}}$ are called the critical binomials of $\a$.  Clearly $c_i$ are uniquely determined by $\a$ and $\a$ minimally generates the semigroup $\sa$ if and only if $c_i\ge 2$ for all $i$.  
 
The rows of a principal matrix are relatively prime.  If $n>2$, then the principal matrix has rank at least $2$.  For, if the matrix has rank one, then all the rows are multiples of the first row.  But then, all the entries on the diagonal are negative and the entries off the diagonal are non negative. This is impossible if $n>2$.  
\vskip .2truein

Although the diagonal of a principal matrix is unique, the rest of the entries can be different. In fact, it is possible for $\a$ to have two principal matrices of different ranks.  
\begin{example} \rm{
$\begin{pmatrix}
-3&2&0&0&0\\
3&-2&0&0&0\\
0&0&-3&2&0\\
0&0&3&-2&0\\
5&0&1&0&-4\\
\end{pmatrix}
$
and $\begin{pmatrix}
-3&2&0&0&0\\
3&-2&0&0&0\\
2&0&-3&0&1\\
0&0&3&-2&0\\
5&0&1&0&-4\\
\end{pmatrix}
$ 

are both principal matrices for $\{22,33,26,39,34\}$ where one of them has rank $3$ and another has rank $4$.
}
\end{example}
In fact, any $\{a\{2,3\}\sqcup b\{2,3\} \sqcup \{2c\}\}$ where $3b\in \langle{a,c}\rangle$ will exhibit this phenomenon.  Nevertheless, the principal matrix does say something about the nature of the set $\a$. 

It's not necessary that any matrix of the above form (which has negative diagonal entries and positive off-diagonal entries) is a principal matrix associated with a semigroup. 
\begin{example}\label {pseudo} 
\rm{$A=\begin{bmatrix} -4 & 0 & 1 & 1\\1 & -5 & 4 & 0\\0 & 4 &-5 &1\\3 & 1 & 0 & -2 \end{bmatrix}$ looks like a Principal matrix.  
Taking the gcd of the first column of its adjoint, we see that $A\begin{bmatrix}
7\\
11\\
12\\
16\\
\end{bmatrix} = 0$.  But then, $A$ is not principal because the second row should be $[3,-3,1,0]$ and $-3>-5$.
}
\end{example}

We will call such matrices pseudo principal. 
\begin{definition}
An $n\times n$ integral matrix $A=(a_{ij})$ is called pseudo principal matrix, if $a_{ij} \ge 0, i\neq j, a_{ii}<0, 1\le i,j\le n$ and there exist positive integers $ a_1, \ldots, a_n$ such that $A\begin {bmatrix}
a_1\\
\vdots\\
a_n\\
\end{bmatrix}= 0$.  We say $A = P(\a)$, to denote that $A$ is a pseudo principal matrix with $A\a= 0$ for $\a =\{ a_1, \ldots a_n\} \subset \N$.  
\end{definition}

 For a pseudo principal matrix to be principal,  $a_{ii} $ needs to be minimal with respect to $a_{ii}a_i = \sum _{j\neq i}a_{ij}a_j$ and $\a$ minimally generate the semigroup $\sa$. In an actual principal matrix $a_{ii} <-1$, for all $i$.  That is the reason we require $a_{ii}<-1$ in a pseudo principal matrix.  However, many of our results here will go through even if $a_{ii} = -1$.  
 It is true that if $A$ is a pseudo principal matrix none of whose entries are zeros and whose product with $\mathbf{1}^{T}$ is zero, then $A^{T}$ is a pseudo principal matrix (see \cite{Villarrealdegree} for details).  We will prove in section 4, theorem \ref{pseudoisreal} a condition for a pseudo principal matrix to be principal which allows some zero entries.   
 It is not a serious restriction that the matrix is integral, for if it is over rational numbers, we can clear the denominator and then the matrix will satisfy the same conditions required to be a pseudo principal matrix. Although we begin this study looking for the structures of principal matrices, many of the results seem to be good for the general pseudo principal matrices.  

Given a set  $\{a_1, \ldots, a_n \}$  of relatively prime positive integers, we denote by $\a = 
\begin {bmatrix}
a_1\\
\vdots\\
a_n\\
\end{bmatrix}$. We denote the Principal matrix of $\a$ by $D(\a)$ .

Principal matrices of numerical semigroups of embedding dimension $3$ are completely characterized by the results of Herzog \cite{He}.  Principal matrices of $\a^T=[a_1, a_2, a_3]$ are either
$D(\a) =\begin{bmatrix}
-c_1&c_2&0\\
c_1&-c_2&0\\
a_{31}&a_{32}&-c_3\\
\end{bmatrix}$ and $\sa $ is a complete intersection or it is an almost complete intersection such that the principal matrix with all the $2\times 2$ minors are not zero. In the former case, $\{a_1,a_2\}\sqcup a_3 = d\{c_2, c_1\}\sqcup a_3$ and $a_3= a_{31}c_1+a_{32}c_2$. 

We will generalize this in theorems \ref{structure}, \ref{rkn-1} and \ref{pseudoisreal}.  

As another generalization of this result, we have the notion of gluing. 
\begin{definition} Let $\b=\{b_1, \ldots, b_s\}, \c= \{c_1, \ldots, c_{n-s}\}$ minimally generate $\sb$ and $\sc$ respectively.  Then we say $\sa $ is a gluing of $\sb$ and $\sc$ if there are relatively prime positive integers $d$ and $e$ such that $\a = d\{\b\} \sqcup e\{\c\}$,   $d\in \sc, d \notin \c, e \in \sb, e \notin \b$.  
\end{definition} 
When both $\b$ and $\c$ have more than one element, the concatenation of the principal matrices $D(\b)$ and $D(\c)$ may not always be principal matrix of $D(\a)$.  It is indeed a pseudo principal matrix of $\a$.  When one of them, say $\c$ has exactly one element, it is called a simple split and the principal matrix can reflect this, as can be seen from  Theorem \ref{rkn-1}.  The concept of gluing for numerical semigroups has been introduced in \cite{sgBook} and studied for instance in \cite{jpaa19}.  A numerical semigroup is a complete intersection if and only if it is the gluing of two complete intersection numerical semigroups by the result of Delorme \cite{De}. An $n\times n$ principal matrix of a glued semigroup can have maximal rank $n-1$, see Example \ref{ex}.\ref{exprincrankn-1}.  

\begin{lemma}\label{Elim}
 Let $A=(a_{ij})$ be any $m\times n$.  Let $\psi(t) = \Pi_{i=1}^t \Delta_{[i]} $.  
Suppose $\Delta_{[i]} \neq 0,$ for all $i<r$. Then there exists a series  of row operations on $A$ to arrive at

$\hat{A}(r)= \begin {bmatrix}
 a_{11}&a_{12}&\ldots &a_{1q} &\ldots\\
  0&\Delta_{12}&\ldots & A^{12}_{1q} &\ldots\\
  \vdots&\vdots&\vdots&\vdots&\vdots\\
  0&0&\ldots  \Delta_{[r]}&\ldots A^{[r]}_{[r-1],q}&\ldots\\
  \vdots&\vdots&\vdots&\vdots&\vdots\\
  0&0&\ldots A^{[r-1],i}_{[r]}&\ldots A^{[r-1],i}_{[r-1],q}&\ldots\\
  \vdots&\vdots&\vdots&\vdots&\vdots\\
  0&0&\ldots A^{[r-1],n}_{[r]}&\ldots A^{[r-1],n}_{[r-1],q}&\ldots\\
  \end{bmatrix}$.
  If $A$ is a matrix of integers, then so is $\hat{A}(r)$. 
  
 In particular,  if  $A$ is a pseudo principal matrix, with $\Delta_{[i]}\neq 0$, for all $i<r$, then there is an invertible matrix $M$ with $MA=
\hat{A}(r)= \begin {bmatrix}
-c_1&a_{12}&\ldots &a_{1q} &\ldots\\
  0&\Delta_{12}&\ldots & A^{12}_{1q} &\ldots\\
  \vdots&\vdots&\vdots&\vdots&\vdots\\
  0&0&\ldots  \Delta_{[r]}&\ldots A^{[r]}_{[r-1],q}&\ldots\\
  \vdots&\vdots&\vdots&\vdots&\vdots\\
  0&0&\ldots A^{[r-1],i}_{[r]}&\ldots A^{[r-1],i}_{[r-1],q}&\ldots\\
  \vdots&\vdots&\vdots&\vdots&\vdots\\
  0&0&\ldots A^{[r-1],n}_{[r]}&\ldots A^{[r-1],n}_{[r-1],q}&\ldots\\
  \end{bmatrix}$.

  Further, if $I= \{1,\dots,s\}=[s]$, then 
  
  $\Delta_I\hat{A}(r)=\Delta _{[s]} \hat{A}(r)= \Pi _{i=1}^{s}\Delta _{[i]}(A)$ for $s\le r$ and 
  $\Delta _{[s]}\hat{A}(r) = \alpha(s) \Delta_{[s]}, \alpha(s) \neq0$,
  \end{lemma}

 \begin{proof}
We proceed with row operations as in Gaussian elimination, except we will not   do any row interchanges. Indeed, we will not need to do so, thanks to the hypothesis $\Delta_{[i]} \neq 0, i<r$. 
 For $i\ge 2$,  by adding $-a_{i1}/a_{11}$ times the first row to the $i$th row and then multiplying the $i$th  row by $a_{11}$,  we arrive at the matrix 
$\hat{A'}(1)
 = \begin{bmatrix}
  a_{11} &a_{12}&\ldots &a_{1q} &\ldots\\
  0& \Delta_{12}&\ldots &  A^{12}_{1q} &\ldots\\
  \vdots&\vdots&\vdots&\vdots&\vdots\\
  0& A^{1i}_{12}&\ldots& A^{1i}_{1q}\ldots\\
   0& A^{1n}_{12}&\ldots& A^{1n}_{1q}\ldots\\
  \end{bmatrix}$.
  Suppose we have obtained by induction, 
  
  $\hat{A'}(r-1) = \begin{bmatrix}
  a_{11}&a_{12}&\ldots &a_{1q} &\ldots\\
  0&\Delta_{12}&\ldots & A^{12}_{1q} &\ldots\\
   0&0& \Delta_{[p]}&\ldots A^{[p]}_{[p-1],q}&\ldots\\
  \vdots&\vdots&\vdots&\vdots&\vdots\\
  0&0&\ldots \Delta_{[r-1]}&\ldots A^{[r-1]}_{[r-2],q}&\ldots\\
  \vdots&\vdots&\vdots&\vdots&\vdots\\
  0&0&\ldots  A^{[r-2],i}_{[r-1]}&\ldots A^{[r-2],i}_{[r-2],q}&\ldots\\
  \vdots&\vdots&\vdots&\vdots&\vdots\\
  0&0&\ldots  A^{[r-2],n}_{[r-1]}&\ldots A^{[r-2],n}_{[r-2],q}&\ldots\\
  \end{bmatrix}$

 Now, since $\Delta_{[r-1]} \neq 0$, we can continue the row operations to make the entries in the $r-1$st column below.
   
 We use the following relations, which can be verified by direct computation of determinants and say, Plucker relations, 
   
   $$A^{[r-1]}_{[r-1]}A^{[r-2],k}_{[r-2],q}-A^{[r-2],k}_{[r-1]}A^{[r-2],r-1}_{[r-2],q}= A^{[r-2]}_{[r-2]}A^{[r-2],r,k}_{[r-2],r,q}, k,q \ge r -1$$
   
We add a proof in the appendix. 
   
   Then we get the matrix,

$\hat{A}(r)= \begin {bmatrix}
 a_{11}&a_{12}&\ldots &&a_{1q} &\ldots\\
  0&\Delta_{12}&\ldots && A^{12}_{1q} &\ldots\\
  \vdots&\vdots&\vdots&&\vdots&\vdots\\
  0&0&\ldots \Delta_{[r-1]}&&\ldots A^{[r-1]}_{[r-2],q}&\ldots\\
  0&0&\ldots  0&\Delta_{[r]}&\ldots A^{[r]}_{[r-1],q}&\ldots\\
  \vdots&\vdots&\vdots&&\vdots&\vdots\\
  0&0&\ldots 0&A^{[r-1],i}_{[r]}&\ldots A^{[r-1],i}_{[r-1],q}&\ldots\\
  \vdots&\vdots&\vdots&\vdots&\vdots\\
  0&0&\ldots 0&A^{[r-1],n}_{[r]}&\ldots A^{[r-1],n}_{[r-1],q}&\ldots\\
  \end{bmatrix}$

  Finally, in the row operations, we did not permute the rows at all and every time we multiplied the $s$th row onwards by $\Delta _{[s-1]}$ and divided  by $\Delta _{[s-2]}$.  Hence comparing the principal minors, we get 
  
  $$\Delta _{[s]}\hat{A}(r) = \Pi_{i=1}^{s}\Delta _{[i]}, s\le r$$ 
  
  and  $$\Delta _{[s]}\hat{A}(r) = \alpha(s) \Delta_{[s]}, \alpha(s) \neq0$$

 \end{proof}

 \begin{theorem} \label{sign}  
 Let $A = P(\a) $ be a pseudo principal matrix. 
Then, 

1. $A^{I,p}_{I,q}\ge 0, p\neq q$ if and only if $|I| $ is even.  

2.  If $\Delta_I = 0$, then $\Delta_{I,j} = 0$ for all $j$.

3. $\Delta_I \ge 0$ if and only if $|I|$ is even. 
 
\end{theorem}
\begin{proof}

The proof is by induction on $|I|$.  
In what follows, we will take the number $0$ to be  both  positive and negative. 

Without loss of generality, we can assume $I = \{1,\ldots, t\}$ for some $t= |I|$. 

Suppose $|I| = 1$.  Then $\Delta _I = -c_I<0$ and $A^{I,j}_{I,k} = \det \begin{pmatrix} -c_I& a_{Ik}\\a_{jI}& a_{jk} 
\end{pmatrix} \le 0$.  

Let $|I|= 2$ so that $I = \{ 1,2\}$.  Since $\Delta_1 \neq 0$, so after a row operation, of adding $c_1^{-1}$ times the first row to the kth row, and then multiply the rows by $-c_1$, we get to
$\hat{A}(2) = \begin{bmatrix}
-c_1& a_{12}& a_{13} & \ldots\\
0& \Delta_{12} & A^{12}_{13} & \ldots \\
\vdots &\vdots& \vdots &\vdots\\
\end{bmatrix}$

Since $\hat{A}(2)\a = 0$ with $\a\in \Z_+$ and $A^{1,2}_{1,j} \le 0, j\ge 3$, we conclude $\Delta_{12} \ge 0$.
So, $\Delta_I \ge 0, |I| = 2$.   

By induction, we assume 

1. Sign $A^{I,j}_{I,k} = (-1)^{|I|}, j\neq k \notin I$, $|I|\le r-2$

2.  If $\Delta_I = 0$, then $\Delta_{I,j} = 0$ for all $j$ for $|I| \le r-1$

3. Sign $\Delta _I = (-1)^{|I|}$,  $|I|\le r-1$

We will prove 1 for $r-1$ and 2 and 3 for $r$.  

Let $|I|= r-1$. Again, we may assume $I = \{ 1,2, \ldots, r-1\}$.  

Consider $A^{I,k}_{I,j}$.  Expanding along the $r$th column which is $\{ a_{k1}, \ldots, a_{kr-1}, a_{kj}\}$, we get, 

$$A^{I,k}_{I,j} = a_{kj} \Delta _I + \sum _{i=1}^{r-1} a_{ki} (-1)^{r-i} A^{([r-1]- i),i}_{([r-1]- i),j} (-1)^{r-1-i}$$
$$= a_{kj} \Delta _I - \sum _{i=1}^{r-1} a_{ki}   A^{([r-1]-i) , i}_{([r-1]- i), j}  $$

Now, sign of $a_{kj}\Delta _I = (-1)^{r-1}$ and 
sign of $-A^{{(I- i) ,i}}_{(I-i),j} = (-1)^{r-1}, j\neq i$

Hence sign $A^{I,k}_{I,j} = (-1)^{|I|}, j\neq k$, as desired. 

Next suppose $\Delta_{[t]} = 0$ for some   $t\le r$.  Let $s$ be the smallest  number such that $\Delta_{[s]} = 0$.  By lemma \ref{Elim}, $\hat{A}(s) \a=0$, for $\a\in (\Z_+)^n$ .  We have already shown that for all $j\neq k\ge s$, the sign of  $A^{[s-1]j}_{[s-1]k} = (-1)^{|I|} $.  Since $\Delta _{[s]} =0$, we must have $A^{[s-1]j}_{[s-1]k}  =0$ and hence the entire $s$th row of $\hat{A}(s)$ is zero.  Since the principal minors of $A$ and $\hat{A}(s)$ are non zero multiples of each other, we must have $\Delta_{[t]}=0, t\ge s$.  Thus, we have shown that $2$ holds for $|I|=r$.  

It remains to prove $3$, for $|I|=r$.  Let $I= \{ 1, 2,\ldots r\}$. If $\Delta _I = 0$, we are done.  If it is not zero, then   $\Delta _{1,2, \dots, r-1}$ is not zero. By lemma \ref{Elim}, $A$ is row equivalent to $\hat{A}(r)$ whose $r$-th row is 
$[0\ldots 0, \Delta_I \ldots A^{1,\ldots r-1, r}_{1, \ldots r-1,j} \ldots]$.  Since the sign of $A^{1,\ldots r-1, r}_{1, \ldots r-1,j}= (-1)^{r-1}$, as we proved above, we get the sign of $\Delta_I = (-1)^r$ and the induction is complete. 

\end{proof}

\begin{remark}\label{vanishing}
As a corollary to the proof of the above theorem, we see that $$\Delta_{[r]} = 0 \implies \Delta_{[t]}=0, t\ge r$$

\end{remark}

\begin{lemma}\label{Elim2}
Let $A= P(\a)$ be a pseudo principal matrix .  Suppose $\Delta_{1,\ldots r-1} \neq 0$. Then there exists a sequence of row operations on $A$ to arrive at $\hat{A}(r) $ whose $(p,q)$th entry is $$\begin{matrix}  
 |A^{[p]}_{[p-1],q}|, p\le r, q> p\\
 -|\Delta_{[p]}|, p=q\le r\\
 0, q\le r-1, q < p \\
  |A^{[r-1],p}_{[r-1],q}|, r\le p,q, p\neq q \\
  -|\Delta{[r-1],p}|, r\le p=q\\
 \end{matrix}$$
 In other words, 
  $$\hat{A}(r)=\begin{bmatrix}
  -c_1&a_{12}&\ldots &a_{1q} &\ldots\\
  0&-|\Delta_{12}|&\ldots &|A^{12}_{1q}|&\ldots\\
  \vdots&\vdots&\vdots&\vdots&\vdots\\
  0&0&\ldots-|\Delta_{[r]}|&\ldots |A^{[r]}_{[r-1],q}|&\ldots\\
  \vdots&\vdots&\vdots&\vdots&\vdots\\
  0&0&\ldots -|\Delta_{[r-1],i}|&\ldots |A^{[r-1],i}_{[r-1],q}|&\ldots\\
  \vdots&\vdots&\vdots&\vdots&\vdots\\
  0&0&\ldots -|\Delta_{[r-1],n}|&\ldots |A^{[r-1],n}_{[r-1],q}|&\ldots\\
  \end{bmatrix}
  $$
  
\end{lemma}
\begin{proof}
 Since $\Delta_{1,\ldots, r-1} \neq 0$, we get from theorem \ref{sign}
$\Delta_{1,..t} \neq 0, t\le r-1$.  Hence by Lemma \ref{Elim}, we can get to an $\hat{A}(r)$. Finally, by theorem \ref{sign},  $A^{1,2,...r-1,p}_{1,2...r-1,q}$ and $A^{1,2,...r-1,p}_{1,2...r-1,p}$ have opposite signs and 
$A^{[p]}_{1,2...r-1,q}$ and $A^{1,2,...r-1,p}_{1,2...r-1,p}$ have opposite signs.  So we arrive at $\hat{A}(r)$ as desired.
\end{proof}
\begin{lemma}\label{Elim3}
Let $A=P(\a)$ be a pseudo principal matrix. Suppose $\Delta_{[r-1]} \neq 0$ and $\Delta_{[r]}=0$.  Then $a_{rq}= 0, q\ge r+1$ and $a_{ri}a_{iq} =0, q\ge r+1$, for all $i <r$. 
\end{lemma}
\begin{proof}
We have $\hat {A}(r)$ as in the lemma \ref {Elim2}.  Since $\hat {A}(r)\a=0$ and $\a $ has positive entries, considering the $r$th  row of $\hat {A}(r)$, we get 

$$ -|\Delta_{[r]}|a_r+ \sum_{q=r+1}^n|A^{[r]}_{[r-1],q}|a_q=0$$

Since $\Delta_{[r]}=0$, 
$$  \sum_{q=r+1}^n|A^{[r]}_{[r-1],q}|a_q=0$$ and hence $$A^{[r]}_{[r-1],q}=0, q\ge r+1$$

But $$A^{[r]}_{[r-1],q}= a_{rq}\Delta_{[r-1]}-\sum_{i=1}^{r-1}a_{ri}A^{[r-1]-i,i}_{[r-1]-i,q}=0.$$

Since all the terms in this sum have the same sign $(-1)^{r-1}$, each of the term is zero.  But $\Delta_{[r-1]}\neq 0$.  So, $a_{rq}= 0, q\ge r+1$. 

Further, if $a_{ri}\neq 0$ for some $i<r$,  then $ A^{[r-1]-i,i}_{[r-1]-i,q}= 0$.  

 Hence applying the above argument, with a reordering as, $1,2,  i-1, i+1, ..r-1,i$, we get $a_{iq}=0, q\ge r+1$.  
\end{proof}
\begin{corollary}\label{minors}
Let $A$ be an $n\times n$ pseudo principal matrix of rank $r$. Then $A$  has a principal minor of size $r$ that is not zero.  If rank $A \le n-2$, then there is a principal minor of size $r$ that is zero. 
\end{corollary}
\begin{proof}
Suppose $t-1$ is the maximal size of a nonzero principal minor of $A$.  After reordering $\a$, we can apply lemma \ref{Elim2} to get $\hat{A}(t)$. In $\hat{A}(t)$, all principal minors of size $t$ must vanish. Hence all the diagonal entries of $\hat{A}(t)$, in rows $\ge t$ are zero.  But by theorem \ref {sign}, this implies the entire rows after $t-1$ are zeros and the rank of $A$ = rank $\hat{A}(t)$ is $t-1$.  Thus, there is always a nonzero maximal minor of size rank $A$. 

\vskip .2truein

Now suppose rank $A = r$ and all principal minors of size $r$ are not zero. Then we get $\hat{A}(r)$ which has nonzero diagonal entries.  
But $$\hat{A}(r) = \left[\begin{matrix}
T_1 & T_2\\
0_{n-r\times r-1} & T_4\\
\end{matrix}
\right]$$
 where $T_1$ is an $r-1 \times r-1$ upper triangular matrix with nonzero diagonal entries.  Hence we must have rank $T_4=1$ and hence every row after $t\ge r$ is a multiple of the $r$-th row.  However, the entries on the diagonal $\Delta_I$ are of opposite signs to the rest of the entries in the row and there is at least one nonzero entry above the diagonal in each row. So, there can only be two nonzero rows in $T_4$.  Hence $r-1+2 = n$ or $r = n-1$. Hence if rank $A \le n-2$, there must be a principal minor of size $r$ that is zero. 
\end{proof}
 \begin{corollary}\label{rank}
 If $A= P(\a)$ or $D(\a)$, then rank $A \ge \frac{n}{2}$.  
 \end{corollary}
 \begin{proof}
 If rank $A$ = $n-1$, it is $\ge \frac{n}{2}$.  Let  $r = \mbox{rank} \ A \le n-2$.   Then by corollary \ref{minors}, there is a nonzero principal minor of size $r$.  Reorder $\a$ to make this the top left minor. So, we have the matrix $\hat{A}(r+1)$.  Since all the $r+1$ minors are zero, we get $A^{1,\ldots r ,p}_{1,..r,q} =0$.  Let $r+1\le p< q$. Then this determinant is the sum of terms of the same sign. Hence they must all be zero.  So, $a_{pq}\Delta_{1,\ldots r}=0$.  But  $\Delta_{1,\ldots r} \neq 0$.  So, $a_{pq} = 0, r+1\le p<q$. Hence 
 $$A = \left[
 \begin{matrix}
 T_1  &T_2\\
 T_3 &T_4\\
 \end{matrix}
  \right],$$ where $T_1$ is $r\times r$ and $T_4$ is $n-r\times n-r$ lower triangular matrix with  negative entries on the diagonal and hence is of rank $n-r$.  So, rank $A \ge n-r$.  So, $r\ge n/2$.
 
 \end{proof}
 
 In the following theorem, we give the structure of a principal matrix of rank less than $n-1$. 
 \vskip .2truein
 \begin{theorem}\label{structure}Let $A=P(\a)$ be an $n\times n$  pseudo principal matrix.   Suppose the rank  $A =n-p\le n-2$.
 Then there exists integers  $r_1, r_2, \ldots r_p \ge 2 $ such that $ \sum r_i\le m$, and after a possible reordering $\a = \r_1\sqcup \r_2 \sqcup  \ldots \sqcup \r_p \sqcup \r_{p+1}$   ,where $|\r_i| = r_i,1\le i\le p$ and 
 $$A = \begin {bmatrix}
 D & \0 \\
 A_1& A_2\\
 \end{bmatrix}$$ 
 where $D$ is the concatenation of $P(\r_i) $, for all $1\le i \le p$.
 We say such a principal matrix $D(\a)$ and the semigroup $<\a>$ is of type $r_1+r_2+ \ldots+r_p+1+1+...$, where $r_1+\ldots +r_p+1+1+...= n$.  If $A= D(\a)$ is principal, then $D$ is the concatenation of $D(\r_i)$, for all $1\le i \le p$.
 \end{theorem}
 \begin{proof}
 Let rank $A = r$.  Since  $r\le n-2$,  there is a principal minor of size $r$ that is zero.  Suppose $t$ is the smallest size of a principal minor that is zero.  By reordering $a_i$ if necessary, we may take  $\Delta_{[t]}=0$ and $\Delta_{[t-1]}\neq 0$. 
 By the lemma \ref{Elim3}, we get $a_{tj} = 0, j\ge t+1$. Now, if $a_{tk} \neq 0, $for some $k<t$, then we get $a_{kj} =0, j\ge t+1$. 
 
By reordering if necessary, we can take $a_{tj}=0, j\le s-1, a_{tj} \neq 0, j\ge s $.  Now, if $a_{ij}\neq 0$, for some $i, s\le i\le t-1, j\le s-1$, then we may relabel $\a$ to make the $i$th row the $t$th row, by the minimality of $t$ and  get $a_{jk}= 0, k\ge t+1$ . So, there exists an $s$, such that $a_{ij} =0, s\le i\le t, j\le s-1, a_{ij} = 0,s\le i\le t, j\ge t+1$.
By looking at the rows $s, \ldots, t$ and $A\a= 0$, we find that the $t-s+1\times t-s+1$ principal minor  $\Delta _{\{s,\ldots ,t\}}= 0$.  Since $t$ is minimal, $s=1$. Let $t= r_1$.  By reversing the order, we conclude that the bottom right $r_1\times r_1$ submatrix of $A$ is a pseudo principal matrix of $\{a_{n-r_1+1}, \ldots, a_n\}$ and a zero matrix of size $r_1\times n-r_1$ as the bottom left.  Then the first $n-r_1+1$ rows of $A$ form a matrix $B$ which again $B$ with $B\a=0$.  We can repeat this with $B$.
Thus, $\a = (a_1, \ldots a_{n-r_1})\sqcup  (a_{n-r_1+1}, \ldots, a_n)$.  Let  $\r_1 = (a_{n-r_1+1}, \ldots, a_n)$.  
$$A= \begin{bmatrix} B_{11} & B_{12}\\
\bf{0}_{r_1\times(n-r_1)}&P(\r_1)\\
\end{bmatrix}$$

Rank of $P(\r_1) = r_1-1$.  So, the $n-r_1 \times n-r_1$ matrix $B_{11}$ has rank $\le n-r_1-2$.  Let $t_1$ be the smallest size of a principal minor in $B_{11}$   which is zero.  Repeating the argument above, and relabeling $\a$, we arrive at 

$$A = \begin{bmatrix}
B_{21}& B_{22}&B_{23}\\
\0&P(\r_2)&\0\\
\0&\0&P(\r_1)\\
\end{bmatrix}$$

Continuing this, we get the result.  In each iteration, we get an $\r_i$ and with the rank $r_i-1$ and we stop when the matrix $B_{i1}$ is invertible. Thus, the rank of $A = n-p$.  If $A=D(\a)$ is principal, then so are $P(\r_i)= D(\r_i)$.
\end{proof}
\begin{remark} \label{pseudopos}\rm{
In fact, we proved the following:
Let $A=(a_{ij})$ be an $r\times n$ matrix with $a_{ii}<0, a_{ij}\ge 0,i\neq j$  with $A\a^T = 0, \a \in (\Z_+)^n$. Suppose $\Delta_{[r]} = 0$ and $\Delta_I\neq 0$, for any $|I|<r$.  Then $A = [T_{r\times r} \ \0_{r\times n-r}]$
}
\end{remark}
When a pseudo principal matrix has rank $n-1$ and there is a vanishing  principal minor of size $n-1$, we have a similar structure. 
\begin{theorem}\label{rkn-1}
Suppose $A=P(\a)$ is an $n\times n$ pseudo principal matrix of rank $n-1$ and that there is a principal minor of size $n-1$ that vanishes.  Then after a possible reordering of $a_i$'s, there is an $s\ge 1$ such that the matrix 

$$A = \begin{bmatrix}
D&\0_{n-s\times s} \\
C_{s\times n-s}&B_{s\times s} \\
\end{bmatrix}$$
where $D$ is of rank $n-s-1$ and is a pseudo principal matrix. If $A$ is a principal matrix, then $D = D(a_{1} \ldots, a_{n-s+1})$
\end{theorem}
\begin{proof}
We may take $\Delta_{[n-1]} = 0, \Delta_{[n-2]}\neq 0$.  So, as in the first part of the proof of theorem \ref{structure}, we get that there is an $s$ such that $a_{ij} =0, s\le i\le n-1, j\le s-1, a_{ij} = 0,s\le i\le n-1, j=n$.
After a reordering, $A$ has the desired form.
\end{proof}

\begin{remark} \label{examples}
In the section on examples, we give examples to show that it is possible to have types $2+3$ and  $2+2+1$ for $5\times 5$ principal matrices  with rank $3$ and $2+4, 2+3+1, 3+3,2+2+1+1$ for $6\times 6$ principal matrix of rank $4$. 
\end{remark}
 In the section, \ref{sufficient} we consider the case of a pseudo principal matrix of rank $n-1$   none of whose principal minors of size $n-1$ vanish and give a sufficient condition when it is principal. 
\section{Embedding Dimension  4 and 5}
 In this section, we consider the lower embedding dimensions .  In the embedding dimension $3$, we have a complete classification of the numerical semigroups as complete intersection or the ideal of $2\times 2$ minors of a $2\times 3$ matrix by \cite{He}.  So, we know the principal matrix of $(a_1, a_2, a_3)$ in a suitable order is either   $A =\begin{bmatrix} -c_{1} & c_2 & 0   \\ c_1&- c_{2}&0\\a_{31}&a_{32}&-c_{3} \end{bmatrix} $ or  $A =\begin{bmatrix} -c_{1} & a_{12}& a_{13} \\ a_{21}&-c_{2}&a_{23}\\a_{31}&a_{32}&-c_{3}\end{bmatrix}$.
 
 When the embedding dimension $n=4$, we either have the rank as $2 $ or $3$ and if the embedding dimension is $n=5$, the rank is ether $4$ or $3$.  In either case if the rank is not $n-1$, it is $n-2$.

Numerical semigroups of embedding dimension 4 have been analyzed in terms of their principal matrices.  Villarael \cite{rafael} has characterized the almost complete intersection numerical semigroups of embedding dimension 4 by their principal matrices and Bresinsky \cite{Be} gave the picture of the principal matrix for the Gorenstein non complete intersections numerical semigroup in embedding dimension 4.  In this section,  for the embedding dimension $n=4$ we prove that the principal matrix $D(\a)$ has rank $n-2=2$ if and only if $\a= (a_1, a_2, a_3, a_4) = d(c_1,c_2)\sqcup e(c_3,c_4)$ and $D(\a)$ is the concatenation of two $2\times 2$ matrices  $D(c_1,c_2)$ and $D(c_3,c_4)$.  In particular, except for a finite number of cases, it must be a gluing of two complete intersections and is therefore a complete intersection.  Further, if the rank is $3$, then one of these possibilities must be true for $A= D(\a)$

1. $A^TX=0$ has a solution in positive integers. 

2. $\a = d\{c_1, c_2\} \sqcup  \{b_1, b_2\}$ with $b_i \in <c_1, c_2>$ 

3. $\a = d\{ b_1, b_2,b_3\} \sqcup b_4$ is a gluing of two numerical semigroups and $\a$ is either a complete intersection or an almost complete intersection. 

For the embedding dimension $n=5$ we prove that the principal matrix $D(\a)$ has rank $n-2=3$ if and only if either $\a= (a_1, a_2, a_3, a_4,a_5) = d(c_1,c_2,c_3)\sqcup e(c_4,c_5)$ and $D(\a)$ has $3+2$ block structure or $\a= (a_1, a_2, a_3, a_4,a_5) = a_{1} \sqcup d(c_2,c_3)\sqcup e(c_4,c_5)$ and $D(\a)$ has $1+2+2$ block structure. Using theorem \ref{rkn-1}, one can also analyze the case of maximal rank 4 as we did in embedding dimension 5.

\begin{theorem} Suppose $\a=\{a_1,a_2,a_3,a_4\}$ are relatively prime positive integers minimally generating a semigroup $<\a>$. Suppose that the principal matrix $A$ of $\a$ has rank $2$.  Then after a permutation of the entries if necessary, 
$\a = d\{c_1,c_2\}\sqcup e\{c_3,c_4\}$ with $(d,e)=(c_1,c_2)=(c_3,c_4) =1$ and the principal matrix has the following block structure
$$A =\begin{bmatrix} -c_{1} & c_2 & 0 &0 \\ c_1&-c_{2}&0&0\\0&0&-c_{3}&c_4\\0&0&c_3&-c_{4}\end{bmatrix}$$
 \end{theorem}
 \begin{proof}
 Let $A=\begin{bmatrix} -c_{1} & a_{12} & a_{13} &a_{14} \\ a_{21}&-c_{2}&a_{23}&a_{24}\\a_{31}&a_{32}&-c_{3}&a_{34}\\a_{41}&a_{42}&a_{43}&-c_{4}\end{bmatrix}$. 
Since the rank of $A$ is $2$,  by corollary \ref{minors}, one of the principal minors of size $2$ must vanish.  By reordering the $a_i$, we may choose $\Delta_{12} = 0$.
By lemma \ref {Elim3}, and the fact that it is a principal matrix, we get $\{a_1, a_2\} = d\{c_1, c_2\}$ where $(c_1, c_2) = 1$. 
$$A =\begin{bmatrix} -c_{1} & c_2 & 0 &0 \\ c_1&-c_{2}&0&0\\*&*&-c_{3}&*\\*&*&*&-c_{4}\end{bmatrix}$$

Now, since the rank of $A$ is $2$, we get that the last two rows must be multiples of each other.  Since $-c_3<0, a_{34} >0, -c_4<0$, the fourth row must be a negative multiple of the third row.  But then $a_{4i} = -xa_{3i},i=1,2$.  So, $a_{31}= a_{32} = a_{41} = a_{42} = 0$.   That implies, $c_3a_3 = a_{34}a_4$.  So, $a_3 = c_4 e, a_4= c_3e$, where $e=(a_3, a_4)$.
Thus, $\{a_1, a_2, a_3, a_4\} = d\{c_2,,c_1\}\sqcup e\{c_4, c_3\}$
\end{proof}
 
 \begin{corollary}
 If $A$ is a $4\times 4$ pseudo principal matrix of rank $2$ , then after some rearrangement of rows and columns, $A$ looks like 
 $$A =\begin{bmatrix} -c_{1} & c_2 & 0 &0 \\ ac_1&-ac_{2}&0&0\\0&0&-c_{3}&c_4\\0&0&bc_3&-bc_{4}\end{bmatrix}$$
 \end{corollary} 
 \begin{proof}
 This is precisely what we proved in the above theorem. For we did not use the fact $A$ is a principal matrix for the concatenation argument. ( We only uses it to get the decomposition of $\a $ as a disjoint union. )
 \end{proof}
 
 \begin{definition}  We say $\a=\{a_1, a_2 \ldots, a_n\}$ is ordinarily decomposable if  $\a=d\b\sqcup e\c$ where $(d,e) = 1$  and $(\b)=(\c)=1$ and the principal matrix $D(\a)$ 
is the concatenation of the principal matrices $D(\b)$ and $D(\c)$.
 \end{definition}
 
 An ordinary decomposition is gluing if $d \in <\c> -\c$ and 
 $e\in <\b>-\b$
 
 What we just proved was if the principal matrix of $\a=(a_1,a_2,a_3,a_4)$ has rank $2$ then it is decomposable as $2$ and $2$. 
 As a corollary to theorem \ref{structure}, we have 
 \begin{corollary} \label{proprankpm}
If $\a$ has no ordinary decomposition, then $A$ has rank $n-1$. 
\end{corollary}
\proof It follows from the Theorem \ref{structure} as $p \geq 2$.
\vskip .1truein
The converse of Corollary \ref{proprankpm} is not true. For an example consider the sequence $\a=\{6,8,10,13\}$. It can be thought of as gluing of $2\{3,4,5\}\sqcup 13\{1\}$. But the rank of $D(\a)=\begin{bmatrix} -3 & 1 & 1 & 0\\1 & -2 & 1& 0\\ 2 & 1 & -2 & 0\\3& 1& 0& -2 \end{bmatrix}$ is $3$. However, when $n=4$,  rank of $D(\a) < 3$ implies it is a decomposition.

 \begin{theorem}\label{bounds}
 Let $C = d\{a_1,a_2\}\sqcup e\{b_1,b_2\}$ be a decomposition with $(d,e) =1$ with the principal matrix of $C$ a concatenation of $D(a_1,a_2)$ and $D(b_1,b_2)$.   Then $d > b_k$, and $e> a_k$ for  $k=1,2$
 \end{theorem}
 \begin{proof}
 Consider $da_1, eb_1, eb_2$.  In what follows, we may assume $(e,a_1) =1$ for otherwise, we can simply divide by the common factor for this triad.  If $da_1$ is not in the semigroup generated by $b_1,b_2$, then this triad is not a complete intersection. This means, the $b_2(eb_1) -b_1(eb_2)$ is not a principal relation.  A contradiction to our assumption on the principal matrix of $C$.  So, $da_1 \in <b_1,b_2>$.  Thus, $da_1 =pb_1+qb_2$ so that  $x_1^{e} -x_3^px_4^q$ is a relation.  But then, $e\ge a_2$.   If $e= a_2$, then since $x_3b_2-x_4b_1$ is a minimal relation, we must have $d\in \langle b_1,b_2 \rangle$.  But then $C$ is not minimal. So, $e>a_2$. Similarly, $e> a_1$ and $d>b_j$ for all $j$. 
 \end{proof}
 \begin{theorem} Suppose $\a=\{a_1,a_2,a_3,a_4,a_5\}$ are relatively prime positive integers minimally generating a semigroup $<\a>$. Suppose that the principal matrix $A$ of $\a$ has rank $3$.  Then after a permutation of the entries if necessary, 
$\a = d\{c_1,c_2, c_3\}\sqcup e\{c_4,c_5\}$ with $(d,e)=(c_1,c_2,c_3)=(c_4,c_5) =1$ and the principal matrix has the following block structures
$$A =\begin{bmatrix} -c_{1} & a_{12} & a_{13} & 0 &0 \\ a_{21} &-c_{2}& a_{23} &0 &0\\a_{31}&a_{32}&-c_{3} &0&0\\ 0 & 0 & 0 &c_{4}&a_{45}\\ 0& 0& 0 &a_{54}&-c_{5}\end{bmatrix}$$ or 
$\a= a_1 \sqcup d\{c_2,c_3\} \sqcup e\{c_4, c_5\}$ with $(c_2, c_3) =(c_4,c_5) = (a_1,d, e) = 1$ and 
$$ A=\begin{bmatrix}-c_{1}&a_{12}&a_{13} &a_{14} &a_{15} \\0&-c_{2}&c_3 &0 &0 \\0&c_2&-c_{3} & 0& 0\\ 0 &0 &0 &-c_{4}&a_{45} \\ 0 & 0&0 & a_{54}&-c_{5} \end{bmatrix}$$
 \end{theorem}
\begin{proof}
Let principal matrix $A=\begin{bmatrix} -c_{1} & a_{12} & a_{13} &a_{14} &a_{15}\\ a_{21}&-c_{2}&a_{23}&a_{24}&a_{25}\\a_{31}&a_{32}&-c_{3}&a_{34}&a_{35}\\a_{41}&a_{42}&a_{43}&-c_{4} &a_{45} \\ a_{51}&a_{52}&a_{53}&a_{54}&-c_{5} \end{bmatrix}$. As the rank of $A=3$, so there is a principal minor of size $3$ that is zero by Corollary \ref{minors}.  If $\Delta _{12}= 0$, then by lemma \ref{Elim3}, we get $a_{1i} = a_{2i} = 0, i=3,4,5$.  Since the rank of A is 3, we must have $\Delta_{3,4,5} = 0$ and $\Delta _{34} \neq 0$.  In either case, we get a $\Delta_{ijk} = 0, \Delta_{ij} \neq 0$.

Now, without loss of generality, let us reorder and assume that $\Delta_{12} \neq 0$ and $\Delta_{123}=0$. 

So, by lemma \ref{Elim3}, we get $a_{3j} = 0, j=4,5$.  Further if $a_{31}a_{32} \neq 0$, then we get $$ A=\begin{bmatrix}-c_{1}&a_{12}&a_{13} &0 &0 \\a_{21}&-c_{2}&a_{23} &0 &0 \\a_{31}&a_{32}&-c_{3} & 0& 0\\ \ast &\ast &\ast &-c_{4}&a_{45} \\ \ast & \ast &\ast & a_{54}&-c_{5} \end{bmatrix}$$

So, the fifth row is a linear combination of the first four rows. This means $a_{54} = -xc_4, -c_5 = xa_{45}$.  So, $\Delta_{45} = 0$.  reordering and applying lemma \ref{Elim3}, we get $a_{4i} = a_{5i}= 0, i= 1,2,3$.
Thus, if $a_{31}a_{32} \neq 0$, $$ A=\begin{bmatrix}-c_{1}&a_{12}&a_{13} &0 &0 \\a_{21}&-c_{2}&a_{23} &0 &0 \\a_{31}&a_{32}&-c_{3} & 0& 0\\ 0 &0 &0 &-c_{4}&a_{45} \\ 0 & 0&0 & a_{54}&-c_{5} \end{bmatrix}$$

Now, if one of $a_{31} $ or $a_{32} = 0$, say $a_{31}= 0$,  then we get $\Delta _{23} = 0$ and hence we come to the case by reordering, where $\Delta _{145}= 0$ and $\Delta_{14} \neq 0$.  If $a_{53}a_{52}\neq0$, we get back to the first case and have the block form above but as concatenation of $D(a_1,a_4,a_5)$ and $D(a_3,a_2)$.  
If one of $a_{52}$ or $a_{53} $ is zero, then after  reordering if necessary, we arrive at

$$ A=\begin{bmatrix}-c_{1}&a_{12}&a_{13} &a_{14} &a_{15} \\0&-c_{2}&c_3 &0 &0 \\0&c_2&-c_{3} & 0& 0\\ 0 &0 &0 &-c_{4}&a_{45} \\ 0 & 0&0 & a_{54}&-c_{5} \end{bmatrix}$$
and $(a_1, a_2, a_3, a_4, a_5) = a_1 \sqcup d\{c_2,c_3\} \sqcup e\{c_4,c_5\}$ where $(c_2,c_3) = (c_4,c_5) = (d,e,a_1) = 1$
\end{proof} 
 \section{Sufficient condition for Principal Matrix}\label{sufficient}

Let $P $ be a pseudo principal matrix with $P\a= 0$.  If $P$ has at least one vanishing principal minor of size $n-1$, then we saw   structures for $\a$ and $P$ in Theorems \ref{structure} and \ref{rkn-1}. In this section we consider the case when none of these principal minors other than $P$ vanish. 

 \begin{lemma}\label{adjoint}
 Let $A$ be a pseudo principal matrix of rank $n-1$ such that $A\a= 0$. None of the principal minors of size $n-1$ are zero if and only if $A^TX=0$ has a solution in positive integers.  
 \end{lemma}
 \begin{proof}
   Suppose none of the principal minors of size $n-1$ are zero. Then the diagonal entries of the $\adj A$ are not zero. Since the columns of the adjoint are multiples of $\a$, we have $\adj A =(a_id_j)_{n \times n}$ .
 By Theorem \ref{sign}, the sign of the diagonal entries are all $(-1)^{n-1}$.  So, we have $d_i$ are all positive or all negative.  In either case, from $\adj A A = \mathbf{0}I$, we get $\sum_{j=1}^n a_id_j a_{jk}= 0$ for all $i$ and $k$.  So, we get $\sum_{j=1}^n|d_j|a_{jk} = 0$ for all $k$ or $[|d_1|,\ldots |d_n|]A=0$, where $|d_i|>0$. Conversely, since rank of $A^T$ is $n-1$, if $A^TX=0$ has a solution in positive integers, then that solution must be a multiple of $d_i$'s.  So, $d_i\neq 0, 1\le i\le n$ which in turn implies none of the $n-1$ minors vanish. 
 \end{proof}
 Using Lemma \ref{adjoint} and Theorem $6.3(c)$ in \cite{Villarrealdegree} we conclude that
 \begin{corollary}
 If $A$ is a pseudo principal matrix of rank $n-1$ such that $A^TX=0$ has solutions in positive integers, then $A^{T}$ is a pseudo principal matrix.  
 \end{corollary}

For a Pseudo principal matrix $A$ of rank $n-1$, there is a $\b \in \N^n$ such that $A\b = 0$. So there is a unique $\a \in \N$ with the entries in $\a$ are relatively prime such that $A\a= 0$.  We denote this by $D^{-1}A = \a$.   
\begin{example}
 Matrix $A=\begin{pmatrix} -4&1&1&1\\2&-4&1&1\\1&1&-3&1\\1&2&1&-3\end{pmatrix}$ is a pseudo principal matrix.  $D^{-1}A =\{20,24,25,31\}$. And $A$ is indeed the principal matrix of  $ \{20,24,25,31\}$.  Thus, in this case, $D(D^{-1}(A)) = A$.  $A^T$ is also pseudo principal but $D^{-1}A^T = \{ 1,1,1,1\}$ and clearly $D(D^{-1}(A^T)) \neq A^T$.
\end{example}

 \begin{theorem}\label{pseudoisreal}
Let $n\ge 3$ and   $A=\begin{bmatrix} -c_{1} & a_{12} &\dots &a_{1n} \\ a_{21} & -c_{2} & \dots &a_{2n} \\ \vdots& \vdots & \dots & \vdots \\ a_{n1}&a_{n2}& \dots &-c_{n} \end{bmatrix}$ be a pseudo princial matrix of rank $n-1$ with $c_i \ge 2$.  Suppose, all the columns add upto zero and the first column of the adjoint is relatively prime and none of the columns in $A$ have more than one zero.   Then  $A$ is the principal matrix of $D^{-1}A$, i.e., $D(D^{-1}A) = A$ and the associated numerical semigroup has embedding dimension $n$. 
\end{theorem}
\begin{proof}
By hypothesis there exist $\a=\begin{bmatrix} a_{1} \\ \vdots \\ a_{n} \end{bmatrix}\in \N^n$ such that $A \a=\0$.  As the rank of $A$ is $n-1$, we can take $\a= D^{-1}A$ and hence $ (a_1, \ldots,a_n) =1$.  So, there exist integers $\lambda_{1},\dots, \lambda_{n}$ such that $\lambda_{1}a_{1}+\dots+\lambda_{n}a_{n}=1$.  Further, by assumption, the column sums of $A$ are zero and hence $\adj A =(-1)^{n-1} [\a, \ldots, \a]$. 

In order to show that the matrix $A$ is the principal matrix $D(\a)$ it is sufficient to show that the relation in each row of $A$ is a principal relation. We show this for the first row and the same proof works for the rest of the rows. 
   
Suppose that $b_{11}a_{1}=\sum_{i=2}^{n}b_{1j}a_{j}$ is a relation with $b_{11} \geq 2$ and $b_{1j} \geq 0, j=2,\dots,n$.  We will show that $b_{11} \geq c_{1}$.   
Let $B=\begin{bmatrix} -b_{11} & b_{12}&\dots&b_{1n}\\a_{21}&-c_{2}&\dots&a_{2n}\\\vdots&\vdots&\dots&\vdots\\a_{n1}&a_{n2}&\dots&-c_{n}\end{bmatrix}$. So, $B\a= 0$ and hence $\det(B)=0$. Therefore, there exist $x_{i},i=1,\dots,n$ such that $\begin{bmatrix} x_1=1 & x_{2} &\dots&x_{n}\end{bmatrix}B=0$. For, $j\ge 2$, let $B_{j}=\begin{bmatrix} -b_{11} & b_{12}&\dots&b_{1n}\\ \vdots & \vdots & \dots & \vdots \\ \lambda_{1}&\lambda_{2}&\dots &\lambda_{n}\\
   \vdots& \vdots &\dots &\vdots \end{bmatrix}$ by replacing the $j$-th row of $B$ by $\lambda_{i}$'s.

Since $B_j \a = \begin{pmatrix}0 \\
   \vdots\\
   1\\
   \vdots\\ 0\end{pmatrix}$,
the $j$ th column of the adjoint of $B_j$  is $\det B_j \begin{bmatrix} a_{1} &a_{2}&\dots&a_{n} \end{bmatrix}^{T}$.

Using $\begin{bmatrix} x_{1} & x_{2} & \dots &x_{n}\end{bmatrix}B=0$ one gets 
\begin{equation*}
\det B_{j}a_{1}= (-1)^{j+1+j-2}\det \begin{vmatrix}  -c_{2} &\dots &\dots &a_{2n}\\ \vdots & \dots & \dots & \vdots \\b_{12}&\dots &\dots &b_{1n}\\ \vdots & \dots & \dots & \vdots \\a_{n2} &\dots &\dots &-c_{n}\end{vmatrix}= x_{j} \det \begin{vmatrix}  -c_{2} &\dots & \dots & a_{2n}\\  \vdots & \dots &\dots &\vdots \\ a_{j2}&\dots&\dots&a_{jn} \\\vdots & \dots &\dots &\vdots  \\a_{n2}&\dots &\dots &-c_{n}  \end{vmatrix}=-(-1)^{n-1}x_ja_1
\end{equation*}

So, $x_j = (-1)^{n-1}\det B_j$ is an integer. 

By Cramer's rule,
$a_1 \det B_j = (-1)^{j+1} B^{\{1,\ldots,j-1,j+1,\ldots,n\}}_{\{2,\ldots,n\}}= (-1)^{j+1}(-1)^{n-2+n-j}B^{\{2 \dots n\}\setminus j,1}_{\{2 \dots n\}\setminus j,j}$.

By theorem \ref{sign}, sign of $B^{\{2 \dots n\}\setminus j,1}_{\{2 \dots n\}\setminus j,j}$ is $n-2$.  We get, the sign of $\det B_j = (-1)^{n-1}$.  Hence, $x_j \ge 0$. 
 
Since $[x_1, \ldots x_n] B=0$, $n\ge 3$, at least three of the $x_j$'s are non zero.  We know, $x_1= 1$. Say, wlog, $x_2,x_3 \neq 0$.  
If $x_j = 0$ for some $j\ge 4$, then $0 =b_{1j}+\sum_{i\ne j, i\ge 2} a_{ij}x_i$. But then $a_{2j}=a_{3j}= 0$, which is impossible. So, $x_j \neq 0$ for any $j$ and are all positive integers. 

Now, $b_{11}= \sum_{i\ge 2} a_{i1}x_i \ge \sum_{i\ge 2} a_{i1} = c_1$.  

Thus, $c_1$ is the smallest possible entry in the diagonal. Similarly,  all of the diagonal entries are the smallest possible and hence $A$ is a principal matrix. 
\end{proof}
\begin{example}
\rm{
Both the conditions on the columns of $A$ in the proposition, \ref {pseudoisreal} are necessary. For instance, 
$\begin{pmatrix}
-5&1&1&4\\
4&-3&0&0\\
1&1&-2&0\\
0&1&1&-4\\
\end{pmatrix}
$ is a pseudo principal matrix, of $(24,32,28, 15)$.   It does have all the columns summing up to zero,  but has two zeros in the last column. It is not principal. For, the first entry $-5$ can be $-4$. Indeed, $4(24) = 3(32)$. 
In the Example \ref{pseudo}, 
$A=\begin{bmatrix} -4 & 0 & 1 & 1\\1 & -5 & 4 & 0\\0 & 4 &-5 &1\\3 & 1 & 0 & -2 \end{bmatrix}$,
we have all the columns sum up to zero, no column with two zeros but then the $n-1=3$ minors of any  of the three rows are not relatively prime. 
}
\end{example}

\section{Examples} \label{secexamples}
We would like to acknowledge that we have used the MonomialAlgebras package \cite{M2} to compute the binomial ideal $I_{\a}$ associated with $\a$. We have written codes in Macaulay2 to compute a principal matrix $D(\a)$  to come up with various examples in this section.
\begin{example} \label{ex} \rm{
\textbf{Embedding Dimension $n=4$}: 
\begin{enumerate}
\item Consider the semigroup $\a=17\{4,7\} \sqcup 13\{3,8\}$. A principal matrix is $D(\a)=\begin{bmatrix} -5 & 1 & 3 & 1\\2&-3&3&1\\ 0&0&-8&3\\0&0&8&-3\end{bmatrix}$ with rank $3$.
This has the minimal number of generators $4$, thus it is an almost complete intersection, not a gluing of the two semigroups $<4,7>$ and $<3,8>$ and hence is not symmetric. 
\item Consider the semigroup  $\a=10\{4,7\} \sqcup 13\{3,8\}$. A principal matrix is $D(\a)=\begin{bmatrix} -7 & 4 & 0 & 0\\7&-4&0&0\\ 0&0&-8&3\\0&0&8&-3\end{bmatrix}$ with rank $2$. Recall by Theorem \ref{bounds}, we must have $d\ge 9, e\ge 8$.  The list of all possible values of $d,e$ such that the corresponding semigroup of the form  $d\{4,7\} \sqcup e\{3,8\}$ is not a gluing is $d=\{10,13\}$ and $e=\{9,10,13,17\}$. The values of $d,e$ such that the corresponding semigroup $d\{4,7\} \sqcup e\{3,8\}$ (not guling) has the same principal matrix of rank $2$ like $D(a)$ mentioned above are as follows. $(d,e)=\{(10,9),(10,13),(10,17),(13,9),(13,17)\}$.

\item \label{exprincrankn-1} Consider the semigroups  of the form  $d\{5,8\} \sqcup e\{3,7\}$. 
If $d= 11, e= 17$, then the semigroup $\a=11\{5,8\} \sqcup 17\{3,7\}$ has a  principal matrix  $D(\a)=\begin{bmatrix} -8 & 5 & 0 & 0 \\ 8 & -5 & 0 & 0\\0 & 0 & -7 & 3\\0 & 0 & 7 & -3\end{bmatrix}$ with rank $2$.  It is not a gluing, the minimal number of generators is $8$. Recall by Theorem \ref{bounds}, we must have $d\ge 8, e\ge 9$. The list of all possible values of $d,e$ such that the corresponding semigroup is not a gluing is $d=\{8,11\}$ and $e=\{9,11,12,14,17,19,22,27\}$. The values of $d,e$ such that the corresponding semigroup $d\{5,8\} \sqcup e\{3,7\}$ (not guling) has the same principal matrix of rank $2$ like $D(a)$ mentioned above are as follows. $(d,e)=\{(8,17),(8,19),(8,27),(11,14),(11,17),(11,19),(11,27)\}$. An example of a gluing semigroup with same principal matrix is $13\{5,8\} \sqcup 18\{3,7\}$. The semigroup $13\{5,8\} \sqcup 10\{3,7\}$ has a principal matrix $D(\a)=\begin{bmatrix}-2 &0&2&1\\0&-5&1&7\\0&0&-7&3\\0&0&7&-3 \end{bmatrix}$ with rank $3$. It is a gluing and hence a complete intersection.  
\item Consider the semigroup $\a=11\{3,4\} \sqcup 13\{3,4\}$. A principal matrix is $D(\a)=\begin{bmatrix} -4 & 3 & 0 & 0\\4&-3&0&0\\ 0&0&-4&3\\0&0&4&-3\end{bmatrix}$ with rank $2$, symmetric, $3$ generators, thus a complete intersection.
\item Consider the semigroup $\a=7\{2,3\} \sqcup 6\{3,4\} $. This is a gluing of two complete intersections and is therefore a complete intersection and the principal matrix has two possibilities. Either 
$D(\a)=\begin{bmatrix} -3 & 2 & 0 & 0\\3&-2&0&0\\ 0&0&-4&3\\0&0&4&-3\end{bmatrix}$ with rank $2$ or
 $D(\a)=\begin{bmatrix} -3 & 0 & 1 & 1\\0&-2&1&1\\ 0&0&-4&3\\0&0&4&-3 \end{bmatrix}$ with rank $3$.  Since it is symmetric and a complete intersection, the principal matrix does not 
  have the standard form for the symmetric non complete intersection semigroups. 
\end{enumerate}
}
\end{example}
\begin{example} \rm{
\textbf{Embedding Dimension $n=5$}
\begin{enumerate}
\item Consider the semigroup $\a=11\{4,7\} \sqcup 13\{5,7,9\}$. 
A principal matrix is $D(\a)=\begin{bmatrix} -7 & 4 & 0 & 0&0\\7&-4&0&0&0\\ 0&0&-5&1&2\\0&0&1&-2&1\\0&0&4&1&-3\end{bmatrix}$ with rank $3$, and the minimal number of generators is $12$.

\item Consider the semigroup $\a=15\{2,3\} \sqcup 14\{2,5\} \sqcup \{43\}$. 
A principal matrix is $D(\a)=\begin{bmatrix} -3 & 2 & 0 & 0&0\\3&-2&0&0&0\\ 0&0&-5&2&0\\0&0&5&-2&0\\1&0&2&0&-2\end{bmatrix}$ with rank $3$, and the minimal number of generators is $7$.

\item Consider the semigroup $\a=14\{2,3\} \sqcup 12\{3,4\} \sqcup \{43\}$.
A principal matrix is $D(\a)=\begin{bmatrix} -3 & 2 & 0 & 0&0\\3&-2&0&0&0\\ 0&0&-4&3&0\\0&0&4&-3&0\\1&0&0&3&-4\end{bmatrix}$ with rank $3$ and the principal matrix has the expected form for rank 3.
Another principal matrix is $D(\a)=\begin{bmatrix} -3 & 0 & 1 & 1&0\\0&-2&1&1&0\\ 0&0&-4&3&0\\0&0&4&-3&0\\1&0&0&3&-4\end{bmatrix}$ with rank $4$.

This is an example where principal matrix is not only not unique but the rank also varies.  
\end{enumerate}
}
\end{example}
\begin{example}
\rm{
\textbf{Embedding Dimension $n=6$}:
\begin{enumerate}
    \item \textbf{Case $2+2+2$ }: Consider the semigroup $\a=11\{2,3\} \sqcup 13\{2,3\} \sqcup 17\{2,3\}$. A principal matrix is  $D(\a)=\begin{bmatrix} -3&2&0&0&0&0\\ 3&-2&0&0&0&0\\ 0&0&-3&2&0&0\\0&0&3&-2&0&0\\0&0&0&0&-3&2\\0&0&0&0&3&-2 \end{bmatrix}$  with rank $3$.  It also has another principal matrix $D(\a)=\begin{bmatrix} -3&2&0&0&0&0\\ 3&-2&0&0&0&0\\ 2&0&-3&0&1&0\\2&0&0&-2&1&0\\0&0&0&0&-3&2\\0&0&0&0&3&-2 \end{bmatrix}$ with rank $4$.  The minimal number of generators $10$. 
    \item \textbf{Case $2+1+1+2$}: Consider the semigroup $\a=\{10,15,14,21,18,27\}$. A principal matrix is $D(\a)=\begin{bmatrix} -3&2&0&0&0&0\\ 3&-2&0&0&0&0\\ 1&0&-2&0&1&0\\0&1&0&-2&0&1\\0&0&0&0&-3&2\\0&0&0&0&3&-2 \end{bmatrix}$ with rank 4 and the minimal number of generators $12$.
    It is easy to generate such examples.  If we simply find three suitable larger) numbers $a, b, c,$ with $3a\in <b,c>$, then $\a=b\{2,3\} \sqcup a\{2,3\} \sqcup c\{2,3\}$ 
    should have two such principal matrices. 
    \item \textbf{case $4+2$}: Consider the semigroup $\a=7\{5,6,7,9\} \sqcup \ 9\{4,6\}$. A principal matrix is  $D(\a)=$ \small \\ $\begin{bmatrix} -3&1&0&1&0&0\\ 1&-2&1&0&0&0\\ 1&0&-2&1&0&0\\0&0&0&-2&2&1\\0&0&0&0&-3&2\\0&0&0&0&3&-2 \end{bmatrix}$ \normalsize with rank $5$. Another principal matrix is \small $D(\a)=\begin{bmatrix} -3&1&0&1&0&0\\ 1&-2&1&0&0&0\\ 1&0&-2&1&0&0\\1&1&1&-2&0&0\\0&0&0&0&-3&2\\0&0&0&0&3&-2 \end{bmatrix}$ \normalsize with rank $4$. 
    Note that this decomposition does not guarantee for the existence of a principal matrix of rank $4$. For example, consider the semigroup $\a=7\{5,6,7,9\} \sqcup \ 9 \{5,6\}$. A principal matrix is $D(\a)=\begin{bmatrix} -3&1&0&1&0&0\\ 1&-2&1&0&0&0\\ 1&0&-2&1&0&0\\1&1&1&-2&0&0\\0&0&0&2&-4&1\\0&0&0&1&1&-2 \end{bmatrix}$. This semigroup does not have any principal matrix with rank $4$.
    \item \textbf{Case $2+3+1$}: Consider the semigroup $\a=13\{2,3\} \sqcup 7\{5,6,7\} \sqcup 43$. A principal matrix is $D(\a)=\small \begin{bmatrix} -3&1&0&1&0&0\\ 1&-2&1&0&0&0\\ 1&0&-2&1&0&0\\1&1&1&-2&0&0\\0&0&0&0&-3&2\\0&0&0&0&3&-2 \end{bmatrix}$ \normalsize  with rank $5$ and another principal matrix is \small $D(\a)=\begin{bmatrix} -3&2&0&0&0&0\\ 3&-2&0&0&0&0\\ 0&0&-4&1&2&0\\0&0&1&-2&1&0\\1&0&1&0&-3&2\\2&0&1&1&0&-3 \end{bmatrix}$ \normalsize with rank $4$.
    \item \textbf{Case $3+3$}: Consider the semigroup $\a=10\{5,6,7\} \sqcup 11 \{5,6,7\}$. A principal matrix is $D(\a)=$ \\ \small  $\begin{bmatrix} -4&1&2&0&0&0\\ 1&-2&1&0&0&0\\ 2&0&-3&2&0&0\\0&0&0&-4&1&2\\0&0&0&1&-2&1\\0&0&0&3&1&-3 \end{bmatrix}$ \normalsize with rank $5$ and another principal matrix is \small $D(\a)=\begin{bmatrix} -4&1&2&0&0&0\\ 1&-2&1&0&0&0\\ 3&1&-3&0&0&0\\0&0&0&-4&1&2\\0&0&0&1&-2&1\\0&0&0&3&1&-3 \end{bmatrix}$ \normalsize with rank $4$.
\end{enumerate}
}
\end{example}
\section{Appendix}

Let $A = ( a_{ij})$ be an $m\times n$ matrix.  Then for all $k,q $

$$A^{[r-1]}_{[r-1]}A^{[r-2],k}_{[r-2],q}-A^{[r-2],k}_{[r-1]}A^{[r-2],r-1}_{[r-2],q}= A^{[r-2]}_{[r-2]}A^{[r-1], k}_{[r-1],q}$$

\begin{proof}
Both sides are equal to zero if $k<r-1$ or $q< r-1$.  So, we take $k,q \ge r-1$.

 Let $r-1= j$ 
$$A^{[j]}_{[j]}A^{[j-1],k}_{[j-1],q}= (A^{[j-1]}_{[j-1]} a_{jj} + \sum _{i=1}^{j-1} (-1)^{i+j} a_{ij} A^{[j-1] - i, j}_{[j-1]} ) (A^{[j-1]}_{[j-1]}a_{kq}+ \sum _{i=1}^{j-1} (-1)^{i+j} a_{iq} A^{[j-1] - i, k}_{[j-1]}) $$
$$ A^{[j-1],k}_{[j]}A^{[j-1],j}_{[j-1],q} = (A^{[j-1]}_{[j-1]} a_{kj} + \sum _{i=1}^{j-1} (-1)^{i+j} a_{ij} A^{[j-1]- i, k}_{[j-1]} )(A^{[j-1]}_{[j-1]} a_{jq} +\sum _{i=1}^{j-1} (-1)^{i+j} a_{iq} A^{[j-1] - i, j} $$
 
Thus, $$\begin{matrix} A^{[j]}_{[j]}A^{[j-1],k}_{[j-1],q}- A^{[j-1],k}_{[j]}A^{[j-1],j}_{[j-1],q} = \\
  A^{[j-1]}_{[j-1]} [ (a_{jj}a_{kq}-a_{kj}a_{jq})A^{[j-1]}_{[j-1]}+\\
  + \sum_{i=1}^{j-1}(-1)^{i+j}  (a_{jj}a_{iq}-a_{ij}a_{jq})  A^{[j-1]-i, k}_{[j-1]} + \\
  \sum_{i=1}^{j-1}(-1)^{i+j}  (a_{kq}a_{ij}-a_{kj}a_{iq}) A^{[j-1]-i, j}_{[j-1]} +\\
  \sum_{s,t=1}^{j-1} (-1)^{s+t} a_{sj} a_{tq} (A^{[j-1] - s, j}_{[j-1]} A^{[j-1] - t, k}_{[j-1]} - A^{[j-1] - t, j}_{[j-1]} A^{[j-1] -s, k}_{[j-1]} )\\
  \end{matrix}
  $$
  
  But, $A^{[j-1] - s, j}_{[j-1]} A^{[j-1] - t, k}_{[j-1]} - A^{[j-1] - t, j}_{[j-1]} A^{[j-1] -s, k}_{[j-1]} = A^{[j-1]}_{[j-1]} A^{[j-1]-\{s,t\},j,k}_{[j-1]}$ by Plucker relations. 
  
  $$  \sum_{s,t=1}^{j-1} (-1)^{s+t}a_{sj}a_{tq} (A^{[j-1] - s, j}_{[j-1]} A^{[j-1] - t, k}_{[j-1]} - A^{[j-1] - t, j}_{[j-1]} A^{[j-1] -s, k}_{[j-1]} ) 
  =   \sum_{s>t=1}^{j-1} (-1)^{s+t}(a_{sj}a_{tq}-  a_{tj}a_{sq})  A^{[j-1]}_{[j-1]} A^{[j-1]-\{s,t\},j,k}_{[j-1]}$$

  So,
  
   $$\begin{matrix} A^{[j]}_{[j]}A^{[j-1],k}_{[j-1],q}- A^{[j-1],k}_{[j]}A^{[j-1],j}_{[j-1],q} = \\
  =   A^{[j-1]}_{[j-1]} [ (  (a_{jj}a_{kq}-a_{kj}a_{jq})A^{[j-1]}_{[j-1]}\\
  + \sum_{i=1}^{j-1}(-1)^{i+j}  (a_{jj}a_{iq}-a_{ij}a_{jq}) A^{[j-1]-i, k}_{[j-1]} +  
  \sum_{i=1}^{j-1}(-1)^{i+j}  (a_{kq}a_{ij}-a_{kj}a_{iq})A^{[j-1]-i, j}_{[j-1]} +\\
  \sum_{s>t=1}^{j-1} (-1)^{s+t}(a_{sj}a_{tq}-  a_{tj}a_{sq})  A^{[j-1]}_{[j-1]} A^{[j-1]-\{s,t\},j,k}_{[j-1]}\\
  = A^{[j-1]}_{[j-1]} A^{[j],k}_{[j]q} 
  \end{matrix}
$$  
\end{proof}

\end{document}